\documentclass[preprint,12pt]{elsarticle}

\usepackage{amssymb}
\usepackage{bbm}
\usepackage{amsmath}
\usepackage{lineno}
\usepackage[numbers]{natbib}
\usepackage{url}
\usepackage[margin=1.5cm]{geometry}
\usepackage[colorlinks=true]{hyperref}
\usepackage{mathrsfs}

\newtheorem{theorem}{Theorem}
\newtheorem{lemma}[theorem]{Lemma}
\newtheorem{remark}[theorem]{Remark}
\newtheorem{algorithm}[theorem]{Algorithm}
\newtheorem{corollary}[theorem]{Corollary}
\newproof{proof}{Proof}

\begin{document}

\begin{frontmatter}

\title{A Weiszfeld-like algorithm for a Weber location problem constrained to a closed and convex set\tnoteref{support}}

\tnotetext[support]{Research partially supported by the CONICET, the SECYT-UNC and the ANPCyT.}

\author[gat]{Germ\'an A. Torres\corref{cor1}}

\ead{torres@famaf.unc.edu.ar}

\cortext[cor1]{Corresponding author}

\address[gat]{Facultad de Matem\'atica, Astronom\'ia y F\'isica, Universidad Nacional de C\'ordoba, CIEM (CONICET), Medina Allende s/n, Ciudad Universitaria (5000) C\'ordoba, Argentina}

\begin{abstract}
The Weber problem consists of finding a point in $\mathbbm{R}^n$ that minimizes the weighted sum of distances from $m$ points in $\mathbbm{R}^n$ that are not collinear. An application that motivated this problem is the optimal location of facilities in the 2-dimensional case. A classical method to solve the Weber problem, proposed by Weiszfeld in 1937, is based on a fixed point iteration. 

In this work a Weber problem constrained to a closed and convex set is considered. A Weiszfeld-like algorithm, well defined even when an iterate is a vertex, is presented. The iteration function $Q$ that defines the proposed algorithm, is based mainly on an orthogonal projection over the feasible set, combined with the iteration function of a modified Weiszfeld algorithm presented by Vardi and Zhang in 2001. 

It can be seen that the proposed algorithm generates a sequence of feasible iterates that have descent properties. Under certain hypotheses, the limit of this sequence satisfies the KKT optimality conditions, is a fixed point of the iteration function that defines the algorithm, and is the solution of the constrained minimization problem. Numerical experiments confirmed the theoretical results.
\end{abstract}


\end{frontmatter}

\linenumbers

\section{Introduction}
 Let $a^1, \ldots, a^m$ be $m$ distinct points in the space $\mathbbm{R}^n$, called vertices, and positive numbers $w_1, \ldots, w_m$, called weights. The function $f : \mathbbm{R}^n \rightarrow \mathbbm{R}$ defined by
\begin{equation}
\label{equation: weber function}
f( x ) = \sum_{j=1}^m w_j \left\| x - a^j \right\|,
\end{equation} 
is called the Weber function, where $\left\| \cdot \right\|$ denotes the Euclidean norm. It is well-known that this function is not differentiable at the vertices, and strictly convex if the vertices are not collinear (we will assume this hypothesis from now on).

The Weber problem (also known as the Fermat-Weber problem) is to find a point in $\mathbbm{R}^n$ that minimizes the weighted sum of Euclidean distances from the $m$ given points, that is, we have to find the solution of the following unconstrained optimization problem:
\begin{equation} 
\label{equation: unconstrained weber problem} 
\begin{array}{rl}
\displaystyle \mathop{ \mathrm{argmin} }_x & \displaystyle f( x ) \\ \mathrm{subject \, to} & x \in \mathbbm{R}^n.
\end{array}
\end{equation}
This problem has a unique solution $x^u$ in $\mathbbm{R}^n$.

The problem was also stated as a pure mathematical problem by Fermat \cite{wesolowsky1993,krarup1997}, Cavalieri \cite{polya1962}, Steiner \cite{couranthilbert1968}, Fasbender \cite{fasbender1846} and many others. Several solutions, based on geometrical arguments, were proposed by Torricelli and Simpson. In \cite{kupitzmartini1997} historical details and geometric aspects were presented by Kupitz and Martini. In \cite{weber1909} Weber formulated the problem (\ref{equation: unconstrained weber problem}) from an economical point of view. The vertices represent customers or demands, the solution to the problem denotes the location of a new facility, and the weights are costs associated with the interactions between the new facility and the customers.

Among several schemes to solve the Weber location problem (see \cite{chatelon1978,eyster1973,kuhn1967,overton1983}), one of the most popular methods was presented by Weiszfeld in \cite{weiszfeld1937,weiszfeldplastria2009}. The Weiszfeld algorithm is an iterative method based on the first-order necessary conditions for a stationary point of the objective function. 

If we define $T_0 : \mathbbm{R}^n \rightarrow \mathbbm{R}^n$ by:
\begin{equation} 
\label{equation: weiszfeld iteration function}
T_0(x) = \left\{ \begin{array}{ll} \frac{ \displaystyle \sum_{j=1}^m \frac{ w_j a^j }{ \left\| x - a^j \right\| } } { \displaystyle \sum_{j=1}^m \frac{ w_j }{ \left\| x - a^j \right\| } }, & \quad \text{if $x \neq a^1, \ldots, a^m$}, \\ & \\ a^k, & \quad \text{if $x = a^k$, $k = 1, \ldots, m$}, \end{array} \right.
\end{equation} 
the Weiszfeld algorithm is:
\begin{equation} 
\label{equation: weiszfeld algorithm}
x^{(l)} = T_0 \left( x^{(l-1)} \right), \quad l \in \mathbbm{N},
\end{equation}
where $x^{(0)} \in \mathbbm{R}^n$ is a starting point.

The Weiszfeld algorithm (\ref{equation: weiszfeld algorithm}), despite of its simplicity, has a serious problem if some $x^{(l)}$ lands accidentally in a vertex $a^k$, because the algorithm gets stuck at $a^k$, even when $a^k$ is not the solution of (\ref{equation: unconstrained weber problem}). Many authors studied the set of initial points for which the sequence generated by the Weiszfeld algorithm yields in a vertex (see \cite{kuhn1973,chandrasekaran1989,brimberg1995,canovas2002,brimberg2003,bernal2005}). Vardi and Zhang \cite{vardizhang2001} derived a simple but nontrivial modification of the Weiszfeld algorithm in which they solved the problem of landing in a vertex.

Generalizations and new techniques for the Fermat-Weber location problem have been developed in recent years. In \cite{eckhardt1980} Eckhardt applied the Weiszfeld algorithm to generalized Weber problems in Banach spaces. An exact algorithm for a Weber problem with attraction and repulsion was presented by Chen et al. in \cite{chenhansenjaumardtuy1992}. Kaplan and Yang \cite{kaplanyang1997} proved a duality theorem which includes as special cases a great variety of choices of norms in the terms of the Fermat-Weber sum. In \cite{carrizosa1998} Carrizosa et al. studied the so called Regional Weber Problem, which allows the demand not to be concentrated onto a finite set of points, but follows an arbitrary probability measure. In \cite{dreznerwesolowsky2002} Drezner and Wesolowsky studied the case where different $l_p$ norms are used for each demand point. In \cite{jalalkrarup2003} the so called Complementary Problem (the Weber problem with one negative weight) was studied by Jalal and Krarup, and geometrical solutions were given. In \cite{drezner2009} Drezner presented a Weiszfeld-like iterative procedure and convergence is proved if appropriate conditions hold. 

In some practical problems it is necessary to consider barriers (forbidden regions). Barriers were first introduced to location modeling by Katz and Cooper \cite{katzcooper1981}. There exist several heuristic and iterative algorithms for single-facility location problems for distance computations in the presence of barriers (see \cite{anejaparlar1994,buttcavalier1996,bischoffklamroth2007,bischofffleischmannklamroth2009}). In \cite{pfeifferklamroth2008} Pfeiffer and Klamroth presented a unified formulation for problems with barriers and network location problems. A complete reference to barriers in location problems can be found in \cite{klamroth2002}. Barriers can be applied to model real life problems where regions like lakes and mountains are forbidden. 

On the other hand, there are location problems whose solution needs to lie within a closed set. For example, see \cite{schaeferhurter1974} for a discussion of the case when the solution is constrained to be within a maximum distance of each demand point. Drezner and Wesolowsky \cite{dreznerwesolowsky1983} studied the problem of locating an obnoxious facility with rectangular distances ($l_1$ norm), where the facility must lie within some prespecified region (linear constraints). A primal-dual algorithm to deal with the constrained Fermat-Weber problem using mixed norms was developed in \cite{michelotlefebvre1987} by Idrissi et al.. In \cite{hansenpeetersthisse1982} Hansen et al. presented an algorithm for solving the Weber problem when the set of feasible locations is the union of a finite number of convex polygons. In \cite{pilottatorres2011} Pilotta and Torres considered a Weber location problem with box constraints.

Constrained Weber problems arise when we require that the solution is in an area (feasible region) determined by, for example, environmental and/or political reasons. It could be the case for a facility producing dangerous materials that must be installed in a restricted (constrained) area. Another example could be the location of a plant in an industrial zone or of a hospital in a non-polluted area.

In this paper a constrained location problem is considered. An algorithm is proposed to solve the following problem:
\begin{equation} 
\label{equation: constrained weber problem} 
\begin{array}{rl}
\displaystyle \mathop{ \mathrm{argmin} }_x & \displaystyle f( x ) \\ \mathrm{subject \, to} & x \in \Omega,
\end{array}
\end{equation}
where $\Omega$ is a closed and convex set, generalizing the problem formulated in \cite{pilottatorres2011}. Problem (\ref{equation: constrained weber problem}) could be seen as a nonlinear programming problem and solved by standard solvers, but they may fail since the Weber function is not differentiable at the vertices.

It can be proved that problem (\ref{equation: constrained weber problem}) has a unique solution $x^*$, since the function $f$ is strictly convex and $\Omega$ is a closed and convex set. On the other hand, it is well-known that the convex hull of the given vertices $a^1, \ldots, a^m$ contains the solution $x^u$ of the unconstrained Weber problem (see for instance \cite[pp. 100]{kuhn1973}). If $\Omega$ contains the convex hull, both solutions $x^*$ and $x^u$ agree. In other cases, the solution $x^*$ is not necessarily a projection of $x^u$ over $\Omega$ (see \cite{pilottatorres2011}). The algorithm is based basically on a slight variation of an orthogonal projection of the Weiszfeld algorithm presented in \cite{vardizhang2001}, that is well defined even when an iterate coincides with a vertex. Properties of the sequence generated by the proposed algorithm related with the minimization problem \ref{equation: constrained weber problem} will be proved in the following sections.

The paper is structured as follows: Section \ref{section: the modified weiszfeld algorithm} describes the results in \cite{vardizhang2001} in which a modified Weiszfeld algorithm is presented and some notation is introduced. In Section \ref{section: the proposed algorithm} the proposed algorithm is defined. Section \ref{section: some definitions and technical results} is dedicated to definitions and technical lemmas. In Section \ref{section: convergence results} the main results about convergence to optimality are presented. Numerical experiments are considered in Section \ref{section: numerical experiments}. Finally, conclusions are given in Section \ref{section: conclusions}.

Some words about notation. As it was mentioned, we will call $x^u$ the solution of problem (\ref{equation: unconstrained weber problem}) and $x^*$ the solution of problem (\ref{equation: constrained weber problem}). The symbols $\| \cdot \|$ and $\langle \cdot , \cdot \rangle$ will refer to the standard Euclidean norm and standard inner product in $\mathbbm{R}^n$, respectively. For a function $f : \mathbbm{R} \rightarrow \mathbbm{R}$ we will denote by $f^{\prime}(a-)$ the left-hand side derivative at $a$, and by $f^{\prime}(a+)$ the right-hand side derivative at $a$.

\section{The modified Weiszfeld algorithm} 
\label{section: the modified weiszfeld algorithm}
This section reviews the main results presented in \cite{vardizhang2001} in which the authors generalize the Weiszfeld algorithm for the case that an iterate lands on a vertex. From now on, this algorithm will be referred to as the modified Weiszfeld algorithm.

In order to make notation easier, we define the function $A : \mathbbm{R}^n \rightarrow \mathbbm{R}$ by:
\begin{equation} 
\label{equation: weight function}
A(x) = \left\{ \begin{array}{ll} \displaystyle \sum_{j=1}^m \frac{ w_j }{ 2 \left\| x - a^j \right\| }, & \quad \text{if $x \neq a^1, \ldots, a^m$}, \\ & \\ \displaystyle \sum_{\substack{j=1 \\ j \neq k}}^m \frac{ w_j }{ 2 \left\| a^k - a^j \right\| }, & \quad \text{if $x = a^k$, $k = 1, \ldots, m$}.  \end{array} \right. 
\end{equation}
Notice that $A(x) > 0$ for all $x \in \mathbbm{R}^n$. In \cite[pp. 563]{vardizhang2001}, the number $A(a^k)$ was called $A_k$. 

A generalization for the iteration function $T_0$, defined in (\ref{equation: weiszfeld iteration function}), is given by $\widetilde{T} : \mathbbm{R}^n \rightarrow \mathbbm{R}^n$ defined as follows:
\begin{equation} 
\label{equation: generalized weiszfeld iteration function}
\widetilde{T}(x) = \left\{ \begin{array}{ll} \displaystyle \frac{ \displaystyle \sum_{j=1}^m \frac{ w_j a^j }{ \left\| x - a^j \right\| } }{ 2 A(x) }, & \quad \text{if $x \neq a^1, \ldots, a^m$}, \\ & \\ \displaystyle \frac{ \displaystyle \sum_{\substack{j=1 \\ j \neq k}}^m \frac{ w_j a^j }{ \left\| a^k - a^j \right\| } }{ 2 A\left(a^k\right) }, & \quad \text{if $x = a^k$, $k = 1, \ldots, m$}. \end{array} \right.
\end{equation}
Notice that $\widetilde{T}$ coincides with $T_o$ in $\mathbbm{R}^n - \left\{ a^1, \ldots, a^m \right\}$.

Let $\widetilde{R} : \mathbbm{R}^n \rightarrow \mathbbm{R}^n$ and $r : \mathbbm{R}^n \rightarrow \mathbbm{R}$ be:
\begin{eqnarray}
\widetilde{R}(x) & = & \left\{ \begin{array}{ll} \displaystyle \sum_{j=1}^m \frac{ w_j \left( a^j - x \right) }{ \left\| x - a^j \right\| }, & \quad \text{if $x \neq a^1, \ldots, a^m$}, \\ & \\ \displaystyle \sum_{\substack{j=1 \\ j \neq k}}^m \frac{ w_j \left( a^j - a^k \right) }{ \left\| a^k - a^j \right\| }, & \quad \text{if $x = a^k$, $k = 1, \ldots, m$}, \end{array} \right. 
\label{equation: definition of widetilder}
\\
& \nonumber
\\
r(x) & = & \| \widetilde{R}(x) \|, \qquad \forall \, x \in \mathbbm{R}^n. \nonumber
\end{eqnarray}
The function $\widetilde{R}$ generalizes the negative gradient of the Weber function since, for all $x \neq a^1, \ldots, a^m$,
\begin{equation} 
\label{equation: nabla equal to minus widetilder}
\nabla f(x) = - \widetilde{R}(x).
\end{equation}

The following lemma is very easy to prove (see \cite[equation (14)]{vardizhang2001}), and it relates the functionals $\widetilde{T}$ and $\widetilde{R}$.
\begin{lemma} 
\label{lemma: widetilder equal to 2A(widetildet - x)}
For all $x \in \mathbbm{R}^n$ we have $\widetilde{R}(x) = 2 A(x) \left[ \widetilde{T}(x) - x \right]$.
\end{lemma}

If we define $\gamma : \mathbbm{R}^n \rightarrow \mathbbm{R}$ by:
\begin{equation*}
\gamma(x) = \left\{ \begin{array}{ll} 0, & \quad \text{if $x \neq a^1, \ldots, a^m$}, \\ 0, & \quad \text{if $x = a^k$ and $r\left(a^k\right) = 0$ for some $k = 1, \ldots, m$}, \\ \displaystyle w_k / r\left(a^k\right), & \quad \text{if $x = a^k$ and $r\left(a^k\right) \neq 0$ for some $k = 1, \ldots, m$}, \end{array} \right.
\end{equation*}
we can see that $\gamma(x) \geq 0$ for all $x \in \mathbbm{R}^n$.

The modified Weiszfeld algorithm presented in \cite{vardizhang2001} is defined by:
\begin{equation*}
x^{(l)} = T\left( x^{(l-1)} \right), \quad l \in \mathbbm{N},
\end{equation*}
where $x^{(0)} \in \mathbbm{R}^n$ and $T : \mathbbm{R}^n \rightarrow \mathbbm{R}^n$ is given by:
\begin{equation} 
\label{equation: iteration function of the modified algorithm with beta}
T(x) = \left( 1 - \beta(x) \right) \widetilde{T}(x) + \beta(x) x,
\end{equation}
where $\beta : \mathbbm{R}^n \rightarrow \mathbbm{R}$ is defined by $\beta(x) = \min \left\{ 1, \gamma(x) \right\}$.

\begin{remark} 
\label{remark: properties of beta}
\begin{enumerate}
\item[(a)] If $x \neq a^1, \ldots, a^m$, then $\beta(x) = 0$ because $\gamma(x) = 0$. So, we can deduce that $T(x) = \widetilde{T}(x)$. Notice that this fact implies that the functional $T$ is continuous in $\mathbbm{R}^n - \left\{ a^1, \ldots, a^m \right\}$.
\item[(b)] It can be seen that if $a^k \neq x^u$, then $0 < \beta( a^k ) < 1$ (see \cite[pp. 563]{vardizhang2001}).
\item[(c)] From equation (\ref{equation: iteration function of the modified algorithm with beta}) we obtain that $T(x) - x = \left( 1 - \beta(x) \right) \left( \widetilde{T}(x) - x \right)$ for $x \in \mathbbm{R}^n$.
\end{enumerate}
\end{remark}

The main result in \cite[pp. 562]{vardizhang2001} is:
\begin{theorem} 
\label{theorem: main result of modified weiszfeld algorithm}
The following propositions are equivalent:
\begin{enumerate}
\item[(a)] $x = x^u$.
\item[(b)] $T(x) = x$.
\item[(c)] $r(x) \leq \eta(x)$.
\end{enumerate}
where 
\begin{equation*}
\eta(x) = \left\{ \begin{array}{ll} 0, & \quad \text{if $x \neq a^1, \ldots, a^m$}, \\ w_k, & \quad \text{if $x = a^k$, $k = 1, \ldots, m$}. \end{array} \right.
\end{equation*}
\end{theorem}

\section{The proposed algorithm}
\label{section: the proposed algorithm}
This section is dedicated to describe the proposed algorithm, introducing some definitions and remarks.

First of all, we can notice that problem (\ref{equation: constrained weber problem}) has a unique solution, due to the fact that $f$ is a non-negative, strictly convex, and continuous function, $\lim_{\|x \| \to \infty} f(x) = \infty$ and $\Omega$ is closed and convex. 

In order to define the proposed algorithm at the vertices, we will need to determine which points of the segment that joins $a^k$ and $T(a^k)$ are in the feasible set $\Omega$. If $k = 1, \ldots, m$, let the set $\mathcal{S}_k$ be defined by:
\begin{equation*}
\mathcal{S}_k = \left\{ \lambda \in [0,1] : ( 1 - \lambda ) T(a^k) + \lambda a^k \in \Omega \right\}.
\end{equation*}
Notice that $\mathcal{S}_k$ could be equal to the empty set in case that $a^k$ and $T(a^k)$ do not belong to $\Omega$. On the other hand, if $a^k \in \Omega$, then $1 \in \mathcal{S}_k$, which means that $\mathcal{S}_k \neq \emptyset$. Thus, we can define:
\begin{equation*}
\lambda( a^k ) = \inf \mathcal{S}_k, \quad a^k \in \Omega.
\end{equation*}
In case a vertex $a^k$ is not in $\Omega$, there is no need to define the number $\lambda( a^k )$.

In the following lemma, a set of basic properties of $\lambda( a^k )$ are listed:
\begin{lemma} 
\label{lemma: properties of lambda(a^k)}
If $k = 1, \ldots, m$ and $a^k \in \Omega$ then:
\begin{enumerate}
\item[(a)] $\lambda( a^k ) \in [0,1]$.
\item[(b)] If $T(a^k) \in \Omega$ then $\lambda( a^k ) = 0$.
\item[(c)] If $T(a^k) \notin \Omega$ then $\lambda( a^k ) \in (0,1]$.
\end{enumerate}
\end{lemma}

\begin{proof}
The proof of (a) follows from the definition of $\mathcal{S}_k$. If $T(a^k) \in \Omega$, then $0 \in \mathcal{S}_k$, so $\lambda(a^k) = 0$, and this proves (b). Finally, for item (c), let us consider that $T(a^k) \notin \Omega$. Since $\Omega$ is a closed set, there is an entire ball centered at $T(a^k)$ that does not intersect $\Omega$, which implies that there exists $\epsilon$ such that $( 1 - \lambda ) T(a^k) + \lambda a^k \notin \Omega$ for all $\lambda \in [0,\epsilon]$. Thus, $\lambda(a^k) \in (0,1]$ and this concludes the proof. \qed
\end{proof}

Let us call $P_{\Omega} : \mathbbm{R}^n \rightarrow \Omega$ the orthogonal projection over $\Omega$. Since $\Omega$ is a nonempty, closed and convex set, the operator $P_{\Omega}$ is a continuous function \cite[pp. 99]{andreassonevgrafovpatriksson2005}.

We define the iteration function $Q : \Omega \rightarrow \Omega$ by:
\begin{equation} 
\label{equation: iteration function of the proposed algorithm}
Q(x) = \left\{ \begin{array}{ll} P_{\Omega} \circ T(x), & \quad \text{if $x \neq a^1, \ldots, a^m$}, \\ \left( 1 - \lambda( a^k ) \right) T(a^k) + \lambda( a^k ) a^k, & \quad \text{if $x = a^k \in \Omega$, $k = 1, \ldots, m$}. \end{array} \right.
\end{equation}
There will be no need to define $Q$ outside $\Omega$ since the proposed algorithm generates a sequence of feasible points. The iteration function $Q$ at $x \in \Omega$ coincides with the orthogonal projection of $T(x)$ over the feasible set when $x$ is different from the vertices. Only when $x$ is a vertex $a^k$ belonging to $\Omega$, $Q(x)$ is defined as the farthest possible feasible point of the segment that joins $x$ with $T(x)$.

The following remark states some basic properties of the iteration function of the proposed algorithm.
\begin{remark} 
\label{remark: properties of Q}
\item[(a)] If $a^k \in \Omega$ and $T(a^k) \in \Omega$, then $Q(a^k) = T(a^k) = P_{\Omega} \circ T(a^k)$.
\item[(b)] If $a^k \in \Omega$, it can be seen that:
\begin{eqnarray*}
Q(a^k) - a^k & = & \left( 1 - \lambda( a^k ) \right) \left( T(a^k) - a^k \right),
\\
Q(a^k) - T(a^k) & = & - \lambda( a^k ) \left( T(a^k) - a^k \right).
\end{eqnarray*}
\item[(c)] The functional $Q$ is continuous in $\mathbbm{R}^n - \left\{ a^1, \ldots, a^m \right\}$.
\end{remark}

\begin{proof} 
The proofs of (a) and (b) are straightforward. For (c), since $P_{\Omega}$ is continuous in $\mathbbm{R}^n$ (see \cite[pp. 99]{andreassonevgrafovpatriksson2005}) and $T$ is continuous in $\mathbbm{R}^n - \left\{ a^1, \ldots, a^m \right\}$ (see Remark \ref{remark: properties of beta}), we have that $Q$ is continuous in $\mathbbm{R}^n - \left\{ a^1, \ldots, a^m \right\}$. \qed
\end{proof}

The proposed algorithm is described below.
\begin{algorithm} 
\label{algorithm: proposed algorithm}
Let $\Omega \subset \mathbbm{R}^n$ be a closed and convex set. Assume that $x^{(0)} \in \Omega$ is an initial approximation such that $f(x^{(0)}) \leq f(a^j)$ for all $j \in \left\{ 1, \ldots, m \right\}$ and $a^j \in \Omega$. Given $\varepsilon > 0$ a tolerance and $x^{(l-1)} \in \Omega$, do the following steps to compute $x^{(l)}$: 

\

\noindent {\bf Step 1:} Compute:
\begin{equation} 
\label{equation: step 1}
x^{(l)} = Q \left( x^{(l-1)} \right).
\end{equation}

\noindent {\bf Step 2:} Stop the execution if
\begin{equation*}
\left\| x^{(l)} - x^{(l-1)} \right\| < \varepsilon,
\end{equation*} 
and declare $x^{(l)}$ as solution to the problem (\ref{equation: constrained weber problem}). Otherwise return to Step 1.

\end{algorithm}

 From the definition of $Q$ it follows that Algorithm \ref{algorithm: proposed algorithm} generates a sequence of feasible iterates. Also notice that if there are vertices in the feasible set, $x^{(0)}$ can be one of them, for example, a vertex $a_s$ such that $f(a_s) \leq f(a^j)$ for all $a^j \in \Omega$. On the other hand, if there are no vertices in the feasible set, $x^{(0)}$ can be chosen as the projection over $\Omega$ of the null vector.

\section{Some definitions and technical results} 
\label{section: some definitions and technical results}
The purpose of this section is to define some entities and prove technical lemmas that will be important in the proof of the main results.

First of all, we will define some useful operators for making notation easier.
If $\mathcal{A} \subset \left\{ 1, \ldots, n \right\}$, then we define $\| \cdot \|_{\mathcal{A}} : \mathbbm{R}^n \rightarrow \mathbbm{R}$ and $\langle \cdot , \cdot \rangle_{\mathcal{A}} : \mathbbm{R}^n \times \mathbbm{R}^n \rightarrow \mathbbm{R}$ by:
\begin{equation*}
\| x \|_{\mathcal{A}} = \sqrt{ \sum_{j \in \mathcal{A}} x_j^2 }, \qquad \langle x , y \rangle_{\mathcal{A}} = \sum_{j \in \mathcal{A}} x_j y_j. 
\end{equation*}
Notice that, when $\mathcal{A} \subsetneq \left\{ 1, \ldots, n \right\}$, $\| \cdot \|_{\mathcal{A}}$ is not necessarily a norm and $\langle \cdot , \cdot \rangle_{\mathcal{A}}$ is not necessarily an inner product.

According to this definition, if $\mathcal{A}$ and $\mathcal{B}$ are sets such that $\mathcal{A} \cap \mathcal{B} = \emptyset$ and $\mathcal{A} \cup \mathcal{B} = \left\{ 1, \ldots, n \right\}$, it can be seen that:
\begin{eqnarray}
\| x \|^2 & = & \| x \|_{\mathcal{A}}^2 + \| x \|_{\mathcal{B}}^2, 
\label{equation: property norma}
\\
\langle x , y \rangle & = & \langle x , y \rangle_{\mathcal{A}} + \langle x , y \rangle_{\mathcal{B}}, 
\label{equation: property innerproda}
\\
c \langle x , y \rangle_{\mathcal{A}} & = & \langle cx , y \rangle_{\mathcal{A}} = \langle x , cy \rangle_{\mathcal{A}}. 
\label{equation: property cinnerproda}
\end{eqnarray}

For $x \in \Omega$, let us define the following sets of indices:
\begin{eqnarray*}
\mathcal{N}(x) & = & \left\{ k \in \mathbbm{N} : \quad 1 \leq k \leq n, \quad (T(x))_k \neq (Q(x))_k \right\}, 
\\
\mathcal{E}(x) & = & \left\{ k \in \mathbbm{N} : \quad 1 \leq k \leq n, \quad (T(x))_k = (Q(x))_k \right\}, 
\end{eqnarray*}
Notice that for all $x \in \mathbbm{R}^n$ we have that $\mathcal{N}(x) \cap \mathcal{E}(x) = \emptyset$ and $\mathcal{N}(x) \cup \mathcal{E}(x) = \left\{ 1, \ldots, n \right\}$.

Let $\alpha : \Omega \rightarrow \mathbbm{R}^n$ be the following function:
\begin{itemize}
\item If $x \neq a^1, \ldots, a^m$:
\begin{equation} 
\label{equation: definition of alpha1}
\alpha(x) = \displaystyle \sum_{j=1}^m \frac{ w_j }{ \| x - a^j \| } \left[ Q(x) - a^j \right].
\end{equation}
\item If $x = a^k \in \Omega$ for some $k = 1, \ldots, m$:
\begin{equation} 
\label{equation: definition of alpha2}
\alpha(x) = \displaystyle \sum_{\substack{j=1 \\ j \neq k}}^m \frac{ w_j }{ \| a^k - a^j \| } \left[ Q(a^k) - (1 - \beta(a^k)) a^j - \beta(a^k) a^k \right].
\end{equation}
\end{itemize}

It can be seen that the function $\alpha$ is related to the iteration function $Q$ of the proposed algorithm, and the iteration function $T$ of the modified algorithm. 
\begin{lemma} 
\label{lemma: properties of alpha}
If $x \in \Omega$, then $\alpha(x) = 2 A(x) \left[ Q(x) - T(x) \right].$
\end{lemma}

\begin{proof}
If $x \neq a^1, \ldots, a^m$, then:
\begin{eqnarray*}
\alpha(x) & = & \sum_{j=1}^m \frac{ w_j }{ \| x - a^j \| } \left[ Q(x) - a^j \right] = \sum_{j=1}^m \frac{ w_j Q(x) }{ \| x - a^j \| } - \sum_{j=1}^m \frac{ w_j a^j }{ \| x - a^j \| } 
\\
&&
\\
& = & \left( \sum_{j=1}^m \frac{ w_j }{ \| x - a^j \| } \right) \left[ Q(x) - \frac{ \sum_{j=1}^m \frac{ w_j a^j }{ \| x - a^j \| } }{ \sum_{j=1}^m \frac{ w_j }{ \| x - a^j \| } }  \right] = 2 A(x) \left[ Q(x) - \widetilde{T}(x) \right] 
\\
&&
\\
& = & 2 A(x) \left[ Q(x) - T(x) \right].
\end{eqnarray*}
where in the last equalities we have used the definition of $\widetilde{T}$ as in (\ref{equation: generalized weiszfeld iteration function}), and the fact that $\widetilde{T}(x) = T(x)$ due to Remark \ref{remark: properties of beta}.

If $x = a^k$ for some $k = 1, \ldots, m$, we follow a similar procedure than in the previous case. \qed
\end{proof}

Now, we will define auxiliary functions that take into account the projection $P_{\Omega}$ in order to prove a descent property of $f$ (see next sections).
If $x \in \Omega$, we define:
\begin{enumerate}
\item[(a)] $E_x : \mathbbm{R}^n \rightarrow \mathbbm{R}^n$, where:
\begin{equation} 
\label{equation: definition of ex}
\left( E_x(y) \right)_k = \left\{ \begin{array}{ll} (Q(x))_k, & \quad \text{if $k \in \mathcal{N}(x)$}, \\ y_k, & \quad \text{if $k \in \mathcal{E}(x)$}. \end{array} \right.
\end{equation}
\item[(b)] If $\mathcal{E}(x) = \left\{ i_1, \ldots, i_r \right\} \neq \emptyset$ define $P_x : \mathbbm{R}^n \rightarrow \mathbbm{R}^r$ where:
\begin{equation*}
\left( P_x(y) \right)_k = y_{i_k}, \quad k = 1, \ldots, r.
\end{equation*}
\end{enumerate}

A useful property of $E_x$, that follows from the definition, is pointed out in the following remark.
\begin{remark} 
\label{remark: ex o q = q}
If $x \in \Omega$ then $E_x \circ Q(x) = Q(x)$.
\end{remark}

The iteration function $Q$ inherits an important property from the orthogonal projection $P_{\Omega}$.
\begin{lemma} 
\label{lemma: property of Q}
If $x \in \Omega$ we have that $\left\langle Q(x) - x , Q(x) - T(x) \right\rangle \leq 0$.  
\end{lemma}

\begin{proof}
If $x \neq a^1, \ldots, a^m$, then $Q(x) = P_{\Omega} \circ T(x)$. By a property of the orthogonal projection \cite[pp. 93]{andreassonevgrafovpatriksson2005} we have that $\langle Q(x) - x , Q(x) - T(x) \rangle \leq 0$. 

If $x = a^k$ for some $k = 1, \ldots, m$, Remark \ref{remark: properties of Q} and Lemma \ref{lemma: properties of lambda(a^k)} imply:
\begin{equation*}
\left\langle Q(a^k) - a^k , Q(a^k) - T(a^k) \right\rangle = - \lambda( a^k ) \left( 1 - \lambda( a^k ) \right) \left\| T(a^k) - a^k \right\|^2 \leq 0,
\end{equation*}
and this concludes the proof. \qed
\end{proof}

The next technical lemma will help us to save computations in other lemmas. 
\begin{lemma} 
\label{lemma: norm of q - aj}
If $x \in \Omega$, $\mathcal{A}$ is a subset of $\left\{ 1, \ldots, n \right\}$ and $j \in \left\{ 1, \ldots, m \right\}$, then:
\begin{equation*}
\left\| Q(x) - a^j \right\|_{\mathcal{A}}^2 = \left\| x - a^j \right\|_{\mathcal{A}}^2 - \left\| Q(x) - x \right\|_{\mathcal{A}}^2 + 2 \left\langle Q(x) - x, Q(x) - a^j \right\rangle_{\mathcal{A}}. 
\end{equation*}
\end{lemma}

\begin{proof}
If $x \in \Omega$, we have:
\begin{eqnarray*}
\| Q(x) - a^j \|_{\mathcal{A}}^2 & = & \langle Q(x) - a^j , Q(x) - a^j \rangle_{\mathcal{A}} 
\\
& = & \langle Q(x) - x + x - a^j , Q(x) - x + x - a^j \rangle_{\mathcal{A}} 
\\
& = & \| Q(x) - x \|_{\mathcal{A}}^2 + \| x - a^j \|_{\mathcal{A}}^2 + 2 \langle Q(x) - x , x - a^j \rangle_{\mathcal{A}} 
\\
& = & \| Q(x) - x \|_{\mathcal{A}}^2 + \| x - a^j \|_{\mathcal{A}}^2 + 2 \langle Q(x) - x , x - Q(x) \rangle_{\mathcal{A}} 
\\
& + & 2 \langle Q(x) - x , Q(x) - a^j \rangle_{\mathcal{A}} 
\\
& = & \| x - a^j \|_{\mathcal{A}}^2 - \| Q(x) - x \|_{\mathcal{A}}^2 + 2 \langle Q(x) - x , Q(x) - a^j \rangle_{\mathcal{A}}.
\end{eqnarray*}
\qed
\end{proof}

If $x \in \Omega$, let us define $g_x : \mathbbm{R}^n \rightarrow \mathbbm{R}$ by:
\begin{equation} 
\label{equation: definition of gx}
g_x(y) = \left\{ \begin{array}{ll} \displaystyle{\sum_{j=1}^m} \frac{ w_j }{ 2 \| x - a^j \| } \| E_x(y) - a^j \|^2, & \text{if $x \neq a^1, \ldots, a^m$}, \\ & \\ \displaystyle{ \sum_{\substack{j=1 \\j \neq k}}^m} \frac{ w_j }{ 2 \| a^k - a^j \| } \| y - a^j \|^2 + w_k \| y - a^k \|, & \text{if $x = a^k$}, \\ & \text{$k = 1, \ldots, m$}. \end{array} \right.
\end{equation}
The values that $g_x$ assumes at $x$ and $Q(x)$ will play an important role in the proof of a property of the objective function $f$.

\begin{lemma} 
\label{lemma: property of gx(x)}
Let $x \in \Omega$ be.
\begin{enumerate}
\item[(a)] If $x \neq a^1, \ldots, a^m$ then:
\begin{eqnarray*}
g_x(x) & = & \frac{1}{2} f(x) + 2 A(x) \left\langle Q(x) - x , Q(x) - T(x) \right\rangle
\\
& - & A(x) \left\| Q(x) - x \right\|_{\mathcal{N}(x)}^2.
\end{eqnarray*}
\item[(b)] If $x = a^k$ for some $k = 1, \ldots, m$, then $g_{a^k}(a^k) = \frac{1}{2} f(a^k)$. 
\end{enumerate}
\end{lemma}

\begin{proof}
Let us suppose that $x \neq a^1, \ldots, a^m$. By property (\ref{equation: property norma}) and (\ref{equation: definition of ex}), we have for $j = 1, \ldots, m$:
\begin{equation*}
\left\| E_x(x) - a^j \right\|^2 = \left\| x - a^j \right\|_{\mathcal{E}(x)}^2 + \left\| Q(x) - a^j \right\|_{\mathcal{N}(x)}^2.
\end{equation*}
Using Lemma \ref{lemma: norm of q - aj}, we can see that:
\begin{eqnarray*}
g_x(x) & = & \sum_{j=1}^m \frac{ w_j }{ 2 \| x - a^j \| } \left[ \left\| x - a^j \right\|_{\mathcal{E}(x)}^2 + \left\| x - a^j \right\|_{\mathcal{N}(x)}^2 \right.
\\
& - & \left. \left\| Q(x) - x \right\|_{\mathcal{N}(x)}^2 + 2 \left\langle Q(x) - x , Q(x) - a^j \right\rangle_{\mathcal{N}(x)} \right].
\end{eqnarray*}
Due to (\ref{equation: property norma}), the definition of the Weber function $f$, the definition of $A$ as in (\ref{equation: weight function}), the property (\ref{equation: property cinnerproda}) and the definition of $\alpha$ as in (\ref{equation: definition of alpha1}), we obtain:
\begin{equation*}
g_x(x) = \frac{1}{2} f(x) - A(x) \| Q(x) - x \|_{\mathcal{N}(x)}^2 + \langle Q(x) - x , \alpha(x) \rangle_{\mathcal{N}(x)}.
\end{equation*}
By Lemma \ref{lemma: properties of alpha}, the fact that $\left( Q(x) \right)_i = \left( T(x) \right)_i$ for all $i \in \mathcal{E}(x)$ and (\ref{equation: property innerproda}), we get:
\begin{eqnarray*}
g_x(x) & = & \frac{1}{2} f(x) - A(x) \| Q(x) - x \|_{\mathcal{N}(x)}^2 
\\
& + & 2 A(x) \langle Q(x) - x , Q(x) - T(x) \rangle,
\end{eqnarray*}
which concludes the proof of (a).

Now, let us assume that $x = a^k$ for some $k = 1, \ldots, m$. Then:
\begin{equation*}
g_{a^k}(a^k) = \sum_{\substack{j=1 \\ j \neq k}}^m \frac{ w_j }{ 2 \| a^k - a^j \| } \| a^k - a^j \|^2 = \frac{1}{2} \sum_{\substack{j=1 \\ j \neq k}}^m w_j \| a^k - a^j \| = \frac{1}{2} f(a^k).
\end{equation*}
This concludes the proof of (b). \qed
\end{proof}

The number $g_x(Q(x))$ can be computed in the next lemma. 

\begin{lemma} 
\label{lemma: property of gx(qx)}
Let $x \in \Omega$ be.
\begin{enumerate}
\item[(a)] If $x \neq a^1, \ldots, a^m$ then:
\begin{eqnarray*}
g_x(Q(x)) & = & \frac{1}{2} f(x) + 2 A(x) \left\langle Q(x) - x , Q(x) - T(x) \right\rangle 
\\
& - & A(x) \left\| Q(x) - x \right\|^2.
\end{eqnarray*}
\item[(b)] If $x = a^k$ for some $k = 1, \ldots, m$, then
\begin{eqnarray*}
g_{a^k}(Q(a^k)) & = & \frac{1}{2} f(a^k) - A(a^k) \left\| Q(a^k) - a^k \right\|^2
\\
& + & 2 A(a^k) \left\langle Q(a^k) - a^k , Q(a^k) - T(a^k) \right\rangle
\\
& - & 2 \beta(a^k) A(a^k) \left\langle Q(a^k) - a^k , \widetilde{T}(a^k) - a^k \right\rangle 
\\
& + & w_k \left\| Q(a^k) - a^k \right\|.
\end{eqnarray*}
\end{enumerate}
\end{lemma}

\begin{proof}
First, let us consider $x \neq a^1, \ldots, a^m$. Due to Remark \ref{remark: ex o q = q} we have:
\begin{equation*}
g_x(Q(x)) = \sum_{j=1}^m \frac{ w_j }{ 2 \| x - a^j \| } \| Q(x) - a^j \|^2.
\end{equation*}
By Lemma \ref{lemma: norm of q - aj} we obtain:
\begin{eqnarray*}
g_x(Q(x))& = & \sum_{j=1}^m \frac{ w_j }{ 2 \| x - a^j \| } \left[ \| x - a^j \|^2 - \| Q(x) - x \|^2 \right.
\\
& + & \left. 2 \langle Q(x) - x , Q(x) - a^j \rangle \right].
\end{eqnarray*}
Due to the definition of the Weber function $f$, the definition of $A$ as in (\ref{equation: weight function}) and the definition of $\alpha$ as in (\ref{equation: definition of alpha1}), we deduce that:
\begin{equation*}
g_x(Q(x)) = \frac{1}{2} f(x) - A(x) \| Q(x) - x \|^2 + \langle Q(x) - x , \alpha(x) \rangle.
\end{equation*}
By Lemma \ref{lemma: properties of alpha} we get:
\begin{equation*}
g_x(Q(x)) = \frac{1}{2} f(x) - A(x) \| Q(x) - x \|^2 + 2 A(x) \langle Q(x) - x , Q(x) - T(x) \rangle,
\end{equation*}
concluding the proof of (a).

Now, consider $x = a^k$ for some $k = 1, \ldots, m$. Due to (\ref{equation: definition of gx}) we have:
\begin{equation*}
g_{a^k}(Q(a^k)) = \sum_{\substack{j=1 \\ j \neq k}}^m \frac{ w_j }{ 2 \| a^k - a^j \| } \| Q(a^k) - a^j \|^2 + w_k \| Q(a^k) - a^k \|.
\end{equation*}
By Lemma \ref{lemma: norm of q - aj}, the definition of the Weber function $f$ and the definition of $A$ as in (\ref{equation: weight function}) we obtain:
\begin{eqnarray*}
g_{a^k}\left( Q(a^k) \right) & = & \frac{1}{2} f(a^k) - A(a^k) \| Q(a^k) - a^k \|^2 
\\
& + & \left\langle Q(a^k) - a^k , \sum_{\substack{j=1 \\ j \neq k}}^m \frac{ w_j }{ \| a^k - a^j \| } \left[ Q(a^k) - a^j \right] \right\rangle 
\\
& + & w_k \| Q(a^k) - a^k \|.
\end{eqnarray*}
Manipulating algebraically, 
\begin{equation*}
Q(a^k) - a^j = Q(a^k) - ( 1 - \beta(a^k) ) a^j - \beta(a^k) a^k + \beta(a^k) ( a^k - a^j ).
\end{equation*}
Due to the definition of $\alpha$ (see (\ref{equation: definition of alpha2})) and the definition of $\widetilde{R}$ (see (\ref{equation: definition of widetilder})) we get:
\begin{eqnarray*}
g_{a^k} \left( Q(a^k) \right)& = & \frac{1}{2} f(a^k) - A(a^k) \left\| Q(a^k) - a^k \right\|^2 
\\
& + & \left\langle Q(a^k) - a^k , \alpha(a^k) \right\rangle - \beta(a^k) \left\langle Q(a^k) - a^k , \widetilde{R}(a^k) \right\rangle
\\ 
& + & w_k \| Q(a^k) - a^k \|.
\end{eqnarray*}
By Lemma \ref{lemma: widetilder equal to 2A(widetildet - x)} and Lemma \ref{lemma: properties of alpha} we have:
\begin{eqnarray*}
g_{a^k} \left( Q(a^k) \right) & = & \frac{1}{2} f(a^k) - A(a^k) \left\| Q(a^k) - a^k \right\|^2 
\\
& + & 2 A(a^k) \left\langle Q(a^k) - a^k , Q(a^k) - T(a^k) \right\rangle 
\\
& - & 2 A(a^k) \beta(a^k) \left\langle Q(a^k) - a^k , \widetilde{T}(a^k) - a^k \right\rangle + w_k \| Q(a^k) - a^k \|.
\end{eqnarray*}
which concludes the proof. \qed
\end{proof}

The next lemma deals with the last two terms of $g_{a^k}(Q(a^k))$.
\begin{lemma} 
\label{lemma: property of something equal to zero}
If $a^k \in \Omega$ for some $k = 1, \ldots, m$, the number
\begin{equation*}
z = w_k \left\| Q(a^k) - a^k \right\| - 2 A(a^k) \beta(a^k) \left\langle Q(a^k) - a^k , \widetilde{T}(a^k) - a^k \right\rangle,
\end{equation*}
is equal to zero.
\end{lemma}

\begin{proof}
If $Q(a^k) = a^k$ the result is true. So, from now on, let us consider that $Q(a^k) \neq a^k$. First, let us check that $a^k \neq x^u$. In case that $a^k = x^u$, then $T(a^k) = a^k$ by Theorem \ref{theorem: main result of modified weiszfeld algorithm}. Since $a^k \in \Omega$, then $T(a^k) \in \Omega$. By Remark \ref{remark: properties of Q} we have that $Q(a^k) = T(a^k) = a^k$ which is a contradiction.

By Remark \ref{remark: properties of beta}, we have that $\beta(a^k) \in (0,1)$ (since $a^k \neq x^u$) and: 
\begin{equation*}
z = w_k \left\| Q(a^k) - a^k \right\| - \frac{ 2 A(a^k) \beta(a^k) }{ 1 - \beta(a^k) }  \left\langle Q(a^k) - a^k , T(a^k) - a^k \right\rangle.
\end{equation*}
Extracting common factors, using Remarks \ref{remark: properties of beta} and \ref{remark: properties of Q}, the fact that $T(a^k) \neq a^k$ (if $T(a^k) = a^k$ then $a^k = x^u$ by Theorem \ref{theorem: main result of modified weiszfeld algorithm}), and the fact that $\widetilde{T}(a^k) \neq a^k$ (if $\widetilde{T}(a^k) = a^k$ then $T(a^k) = a^k$ by definition (\ref{equation: iteration function of the modified algorithm with beta})) we get that:
\begin{eqnarray*}
z & = & 2 A(a^k) \left\| Q(a^k) - a^k \right\| \left\| \widetilde{T}(a^k) - a^k \right\| \left[ \frac{ w_k }{ 2 A(a^k) \left\| \widetilde{T}(a^k) - a^k \right\| } \right.
\\
& - & \left. \beta(a^k) \left\langle \frac{ ( 1 - \lambda( a^k ) ) ( T(a^k) - a^k ) }{ \left\| ( 1 - \lambda( a^k ) ) ( T(a^k) - a^k ) \right\| } , \frac{ T(a^k) - a^k }{ \left\| T(a^k) - a^k \right\| } \right\rangle \right].
\end{eqnarray*}
Simplifying and using the definition of $\beta(a^k)$ we have that:
\begin{equation*}
z = 2 A(a^k) \left\| Q(a^k) - a^k \right\| \left\| \widetilde{T}(a^k) - a^k \right\| \left[ \beta(a^k) - \beta(a^k) \right] = 0,
\end{equation*}
which concludes the proof. \qed
\end{proof}

The purpose of the next two lemmas is to determine a strict inequality between the functions $g_x$ and $f$ at suitable points. First of all, we have to prove the following result.

\begin{lemma} 
\label{lemma: comparison between gxqx and gxx}
Let $x \in \Omega$ be such that $x \neq Q(x)$.
\begin{enumerate}
\item[(a)] If $x \neq a^1, \ldots, a^m$, then $g_x(Q(x)) \leq g_x(x)$. Besides that, if $\mathcal{E}(x) \neq \emptyset$ and $P_x \circ Q(x) \neq P_x(x)$, then $g_x(Q(x)) < g_x(x)$.
\item[(b)] If $x = a^k$ for some $k = 1, \ldots, m$, then $g_{a^k}(Q(a^k)) < g_{a^k}(a^k)$. 
\end{enumerate}
\end{lemma}

\begin{proof}
If $x \neq a^1, \ldots, a^m$, then $g_x(Q(x)) - g_x(x) = - A(x) \| Q(x) - x \|_{\mathcal{E}(x)}^2 \leq 0$, by Lemma \ref{lemma: property of gx(x)} and Lemma \ref{lemma: property of gx(qx)}. Besides that, if $\mathcal{E}(x) \neq \emptyset$ and $P_x \circ Q(x) \neq P_x(x)$ we deduce that $\| Q(x) - x \|_{\mathcal{E}(x)} \neq 0$. Thus, $g_x(Q(x)) < g_x(x)$.

If $x = a^k$ for some $k = 1, \ldots, m$, by Lemmas \ref{lemma: property of gx(x)}, \ref{lemma: property of gx(qx)} and \ref{lemma: property of something equal to zero} we have:
\begin{eqnarray*}
g_{a^k}(Q(a^k)) - g_{a^k}(a^k) & = & - A(a^k) \left\| Q(a^k) - a^k \right\|^2 
\\
& + & 2 A(a^k) \left\langle Q(a^k) - a^k , Q(a^k) - T(a^k) \right\rangle.
\end{eqnarray*}
Due to Lemma \ref{lemma: property of Q} and the fact that $A > 0$ we obtain:
\begin{equation*}
g_{a^k}(Q(a^k)) - g_{a^k}(a^k) \leq - A(a^k) \| Q(a^k) - a^k \|^2 < 0,
\end{equation*}
and the proof is finished. \qed
\end{proof}

\begin{lemma} 
\label{lemma: property gxqx < fx/2}
Let $x \in \Omega$ be such that $x \neq Q(x)$. Then $g_x(Q(x)) < \frac{1}{2} f(x)$.
\end{lemma}

\begin{proof}
Let us consider the case when $x \neq a^1, \ldots, a^m$. By Lemmas \ref{lemma: property of Q}, \ref{lemma: property of gx(x)} and \ref{lemma: comparison between gxqx and gxx} we have that:
\begin{eqnarray*}
g_x(Q(x)) & \leq & g_x(x) = \frac{1}{2} f(x) + 2 A(x) \left\langle Q(x) - x , Q(x) - T(x) \right\rangle 
\\
& - & A(x) \left\| Q(x) - x \right\|_{\mathcal{N}(x)}^2
\\
& \leq & \frac{1}{2} f(x) - A(x) \left\| Q(x) - x \right\|_{\mathcal{N}(x)}^2.
\end{eqnarray*}
If $\mathcal{E}(x) = \emptyset$, then $\| \cdot \|_{\mathcal{N}(x)} = \| \cdot \|$. Therefore:
\begin{equation*}
g_x(Q(x)) \leq \frac{1}{2} f(x) - A(x) \| Q(x) - x \|^2 < \frac{1}{2} f(x).
\end{equation*}
If $\mathcal{E}(x) \neq \emptyset$ and $P_x \circ Q(x) = P_x(x)$, then there exists an index $i \in \mathcal{N}(x)$ such that $x_i \neq (Q(x))_i$ since $x \neq Q(x)$. Thus, $\| Q(x) - x \|_{\mathcal{N}(x)} \neq 0$, which implies:
\begin{equation*}
g_x(Q(x)) \leq \frac{1}{2} f(x) - A(x) \left\| Q(x) - x \right\|_{\mathcal{N}(x)}^2 < \frac{1}{2} f(x).
\end{equation*}
If $\mathcal{E}(x) \neq \emptyset$ and $P_x \circ Q(x) \neq P_x(x)$, due to Lemmas \ref{lemma: property of Q}, \ref{lemma: property of gx(x)} and \ref{lemma: comparison between gxqx and gxx}, we have that:
\begin{equation*}
g_x(Q(x)) < g_x(x) \leq \frac{1}{2} f(x) - A(x) \left\| Q(x) - x \right\|_{\mathcal{N}(x)}^2 \leq \frac{1}{2} f(x).
\end{equation*}

Now, when $x = a^k$ for some $k = 1, \ldots, m$, $g_{a^k}(Q(a^k)) < g_{a^k}(a^k) = \frac{1}{2} f(a^k)$ due to Lemma \ref{lemma: property of gx(x)} and Lemma \ref{lemma: comparison between gxqx and gxx}. \qed
\end{proof}

The next lemma states an equality that relates the Weber function and $g_x$ at appropriate points when $x \neq a^1, \ldots, a^m$. Besides that, this result will be crucial in the next section.

\begin{lemma} 
\label{lemma: equality of gxqx}
Let $x \neq a^1, \ldots, a^m$ be such that $x \in \Omega$ and $x \neq Q(x)$. Then:
\begin{equation*}
g_x \circ Q(x) = \frac{1}{2} f(x) + \left( f(Q(x)) - f(x) \right) + \delta, \quad \delta \geq 0.
\end{equation*}
\end{lemma}

\begin{proof}
Due to the definition of $g_x$ as in (\ref{equation: definition of gx}) and Remark \ref{remark: ex o q = q} we get that:
\begin{equation*}
g_x \circ Q(x) = \sum_{j=1}^m \frac{ w_j }{ 2 \| x - a^j \| } \| Q(x) - a^j \|^2.
\end{equation*}
Adding and subtracting $\| x - a^j \|$ we have:
\begin{eqnarray*}
g_x \circ Q(x) & = & \sum_{j=1}^m \frac{ w_j }{ 2 \| x - a^j \| } \left[ \left\| x - a^j \right\| + \left( \left\| Q(x) - a^j \right\| - \left\| x - a^j \right\| \right) \right]^2 
\\
& = & \frac{1}{2} \sum_{j=1}^m w_j \left\| x - a^j \right\| + \sum_{j=1}^m w_j \left( \left\| Q(x) - a^j \right\| - \left\| x - a^j \right\| \right) 
\\
& + & \sum_{j=1}^m \frac{ w_j }{ 2 \| x - a^j \| } \left( \left\| Q(x) - a^j \right\| - \left\| x - a^j \right\| \right)^2.
\end{eqnarray*}
Notice that the first term of the last equality is the Weber function (divided by two), and the last term is a non-negative number, so we will define it as $\delta$. So, using the definition of the Weber function in the middle term we obtain:
\begin{equation*}
g_x \circ Q(x) = \frac{1}{2} f(x) + \left( f(Q(x)) - f(x) \right) + \delta.
\end{equation*} 
\qed
\end{proof}

\section{Convergence to optimality results} 
\label{section: convergence results}
This section states the main results about convergence of the sequence $\left\{ x^{(l)} \right\}$ generated by Algorithm \ref{algorithm: proposed algorithm}. The next theorem establishes that if a point $x \in \Omega$ is not a fixed point of the iteration function, then the function $f$ strictly decreases at the next iterate.

\begin{theorem} 
\label{theorem: descent property}
Let $x \in \Omega$ be such that $x \neq Q(x)$. Then $f(Q(x)) < f(x)$.
\end{theorem}

\begin{proof}
Let us consider that $x \neq a^1, \ldots, a^m$. By Lemma \ref{lemma: property gxqx < fx/2}, we have that:
\begin{equation*}
g_x \circ Q(x) < \frac{1}{2} f(x).  
\end{equation*}
By Lemma \ref{lemma: equality of gxqx} we get that:
\begin{equation*}
\frac{1}{2} f(x) + f(Q(x)) - f(x) + \delta < \frac{1}{2} f(x).  
\end{equation*}
Simplifying the last expression we obtain:
\begin{equation*}
f(Q(x)) - f(x) + \delta < 0.
\end{equation*}
Finally, 
\begin{equation*}
f(Q(x)) - f(x) \leq f(Q(x)) - f(x) + \delta < 0.
\end{equation*}
Therefore, $f(Q(x)) < f(x)$.

Now, consider that $x = a^k$ for some $k = 1, \ldots, m$. Following a reasoning similar than in \cite[pp. 564]{vardizhang2001}, using Lemma \ref{lemma: comparison between gxqx and gxx} we have that:
\begin{equation*}
g_{a^k} \circ Q(a^k) - g_{a^k}(a^k) < 0.
\end{equation*}
By definition of $g_{a^k}$ we know that:
\begin{eqnarray*}
g_{a^k} \circ Q(a^k) - g_{a^k}(a^k) & = & w_k \left\| Q(a^k) - a^k \right\| 
\\
& + & \sum_{\substack{j=1 \\ j \neq k}}^m \frac{ w_j }{ 2 \left\| a^k - a^j \right\| } \left( \left\| Q(a^k) - a^j \right\|^2 - \left\| a^k - a^j \right\|^2 \right).
\end{eqnarray*}
Using the fact that $\left( a^2 - b^2 \right)/ (2b) \geq a - b$ for $a = \left\| Q(a^k) - a^j \right\|^2 \geq 0$ and $b = \left\| a^k - a^j \right\|^2 > 0$ we obtain that:
\begin{eqnarray*}
g_{a^k} \circ Q(a^k) - g_{a^k}(a^k) & \geq & w_k \left\| Q(a^k) - a^k \right\| 
\\
& - & \sum_{\substack{j=1 \\ j \neq k}}^m w_j \left\| a^k - a^j \right\| + \sum_{\substack{j=1 \\ j \neq k}}^m w_j \left\| Q(a^k) - a^j \right\|.
\end{eqnarray*}
Rearranging terms we deduce that:
\begin{eqnarray*}
0 & > & g_{a^k} \circ Q(a^k) - g_{a^k}(a^k) 
\\
& = & \sum_{j=1}^m w_j \left\| Q(a^k) - a^j \right\| - \sum_{j=1}^m w_j \left\| a^k - a^j \right\| = f(Q(a^k)) - f(a^k),
\end{eqnarray*}
and the proof is complete. \qed
\end{proof}

\begin{corollary} 
\label{corollary: not increasing sequence}
Let $\left\{ x^{(l)} \right\}$ be the sequence generated by Algorithm \ref{algorithm: proposed algorithm}. Then the sequence $\left\{ f \left( x^{(l)} \right) \right\}$ is not increasing. Even more, each time $x^{(l)} \neq Q \left( x^{(l)} \right)$ the sequence strictly decreases at the next iterate.
\end{corollary}

If the sequence $\left\{ x^{(l)} \right\}$ generated by Algorithm \ref{algorithm: proposed algorithm} were not bounded, then we could choose a subsequence $\left\{ y^{(l)} \right\}$ such that $y^{(l)} \to \infty$. But this implies that $f(y^{(l)}) \to \infty$, which is a contradiction since the sequence $\left\{ f \left( x^{(l)} \right) \right\}$ is not increasing.

\begin{remark}
The sequence $\left\{ x^{(l)} \right\}$ generated by Algorithm \ref{algorithm: proposed algorithm} is bounded. So, there exists a subsequence convergent to a point $x^* \in \Omega$. Hence, $x^*$ is a feasible point.
\end{remark}

Due to the nondifferentiability of $f$ at the vertices $a^1, \ldots, a^m$, we can not use the KKT optimality conditions at $a^k$. Therefore, if $a^k$ and $z$ are in $\Omega$, let us define $G_{a^k}^z : [0,1] \rightarrow \mathbbm{R}$ by:
\begin{equation*}
G_{a^k}^z(t) = f( a^k + t(z-a^k) ).
\end{equation*}
If $a^k \in \Omega$, $z \in \Omega$, $t \in [0,1]$ and $\Omega$ convex, we have that $a^k + t(z-a^k) \in \Omega$. Notice that the right-hand side derivative $G_{a^k}^z(0+)$ (or the directional derivative of $f$ in the direction of $z$) exists (see \cite[pp. 33]{jahn2007introduction}). Besides that,
\begin{equation} 
\label{equation: property of gakzp}
G_{a^k}^{z \ \prime}(0+) = w_k \left\| z - a^k \right\| - \left\langle \widetilde{R}(a^k) , z - a^k \right\rangle.
\end{equation}

The next lemma shows that if we are in a vertex $a^k$, the directional derivative of $f$ at $a^k$ in the direction of $Q(a^k)$ is a descent direction.
\begin{lemma} 
\label{lemma: property of a^k as a minimum}
Let $a^k \in \Omega$ be such that $T(a^k) \notin \Omega$. Then:
\begin{equation} 
\label{equation: property of a^k as a mininum 1}
G_{a^k}^{z \ \prime}(0+) \geq G_{a^k}^{Q(a^k) \ \prime}(0+), \quad \forall \, z \in \left[ a^k, T(a^k) \right],
\end{equation}
where:
\begin{equation} 
\label{equation: property of a^k as a mininum 2}
G_{a^k}^{Q(a^k) \ \prime}(0+) = - 2 \left[ 1 - \beta(a^k) \right] A(a^k) \left\| \widetilde{T}(a^k) - a^k \right\| \left\| Q(a^k) - a^k \right\|. 
\end{equation}
\end{lemma}

\begin{proof}
If $T(a^k) = a^k$ then $T(a^k) \in \Omega$, which is a contradiction. Besides that, if $\widetilde{T}(a^k) = a^k$, we would have that $T(a^k) = a^k$ because of (\ref{equation: iteration function of the modified algorithm with beta}), and again it would be a contradiction. So, we will consider $\widetilde{T}(a^k) \neq a^k$ and $T(a^k) \neq a^k$ for the rest of the proof. Since $T(a^k) \neq a^k$, then $\beta(a^k) \in (0,1)$ (see Remark \ref{remark: properties of beta} and Theorem \ref{theorem: main result of modified weiszfeld algorithm}). 

Let us prove equation (\ref{equation: property of a^k as a mininum 2}) first. Now, by (\ref{equation: property of gakzp}) we can see that:
\begin{equation*}
G_{a^k}^{Q(a^k) \ \prime}(0+) = w_k \left\| Q(a^k) - a^k \right\| - \left\langle \widetilde{R}(a^k) , Q(a^k) - a^k \right\rangle.
\end{equation*}
Notice that if $Q(a^k) = a^k$, equation (\ref{equation: property of a^k as a mininum 2}) holds. So, let us consider from now on that $Q(a^k) \neq a^k$. By using Lemma \ref{lemma: widetilder equal to 2A(widetildet - x)} we replace $\widetilde{R}(a^k)$ and get:
\begin{equation*}
G_{a^k}^{Q(a^k) \ \prime}(0+) = w_k \left\| Q(a^k) - a^k \right\| - 2 A(a^k) \left\langle \widetilde{T}(a^k) - a^k , Q(a^k) - a^k \right\rangle.
\end{equation*}
Extracting common factors and using the definition of $\beta$ when it belongs to $(0,1)$ we obtain:
\begin{eqnarray*}
G_{a^k}^{Q(a^k) \ \prime}(0+) & = & 2 A(a^k) \left\| Q(a^k) - a^k \right\| \left\| \widetilde{T}(a^k) - a^k \right\| \Bigg[ \beta(a^k) 
\\
& - & \left. \left\langle \frac{ \widetilde{T}(a^k) - a^k }{ \left\| \widetilde{T}(a^k) - a^k \right\| } , \frac{ Q(a^k) - a^k }{ \left\| Q(a^k) - a^k \right\| } \right\rangle \right].
\end{eqnarray*}

By Remarks \ref{remark: properties of beta} and \ref{remark: properties of Q} the vectors $Q(a^k) - a^k$ and $\widetilde{T}(a^k) - a^k$ are parallel, so:
\begin{equation*}
G_{a^k}^{Q(a^k) \ \prime}(0+) = 2 A(a^k) \left\| Q(a^k) - a^k \right\| \left\| \widetilde{T}(a^k) - a^k \right\| \left[ \beta(a^k) - 1 \right].
\end{equation*}
which is equivalent to (\ref{equation: property of a^k as a mininum 2}).

Now, let us prove (\ref{equation: property of a^k as a mininum 1}). If $z = a^k$ then $G_{a^k}^z(t) = f(a^k)$ for all $t \in [0,1]$, thus $G_{a^k+}^{a^k \ \prime}(0) = 0$, and therefore the inequality (\ref{equation: property of a^k as a mininum 1}) holds. So, let us assume that $z \neq a^k$ for the rest of the proof. Using (\ref{equation: property of gakzp}) and due to Lemma \ref{lemma: widetilder equal to 2A(widetildet - x)} to replace $\widetilde{R}(a^k)$:
\begin{equation*}
G_{a^k}^{z \ \prime}(0+) = w_k \left\| z - a^k \right\| - 2 A(a^k) \left\langle \widetilde{T}(a^k) - a^k , z - a^k \right\rangle.
\end{equation*}
Extracting common factors:
\begin{eqnarray*}
G_{a^k}^{z \ \prime}(0+) & = & 2 A(a^k) \left\| z - a^k \right\| \left\| \widetilde{T}(a^k) - a^k \right\| \left[ \frac{ w_k }{ 2 A(a^k) \left\| \widetilde{T}(a^k) - a^k \right\| } \right. 
\\
& - & \left. \left\langle \frac{ \widetilde{T}(a^k) - a^k }{ \left\| \widetilde{T}(a^k) - a^k \right\| } , \frac{ z - a^k }{ \left\| z - a^k \right\| } \right\rangle \right].
\end{eqnarray*}
Using the expression for $\beta(a^k) \in (0,1)$ we obtain:
\begin{eqnarray*}
G_{a^k}^{z \ \prime}(0+) & = & 2 A(a^k) \left\| z - a^k \right\| \left\| \widetilde{T}(a^k) - a^k \right\| \Bigg[ \beta(a^k) 
\\
& - & \left. \left\langle \frac{ \widetilde{T}(a^k) - a^k }{ \left\| \widetilde{T}(a^k) - a^k \right\| } , \frac{ z - a^k }{ \left\| z - a^k \right\| } \right\rangle \right].
\end{eqnarray*}
If $z$ belongs to the segment that joins $a^k$ and $T(a^k)$ we have that $z-a^k$ and $\widetilde{T}(a^k) - a^k$ are parallel vectors, then:
\begin{equation*}
G_{a^k}^{z \ \prime}(0+) \geq - 2 \left[ 1 - \beta(a^k) \right] A(a^k) \left\| z - a^k \right\| \left\| \widetilde{T}(a^k) - a^k \right\|. 
\end{equation*}
We can write $z = (1-\lambda)T(a^k) + \lambda a^k$ where $\lambda \in [0,1]$. Therefore:
\begin{equation*}
G_{a^k}^{z \ \prime}(0+) \geq - 2 \left[ 1 - \beta(a^k) \right] A(a^k) \left( 1 - \lambda \right) \left\| T(a^k) - a^k \right\| \left\| \widetilde{T}(a^k) - a^k \right\|. 
\end{equation*}
for all $\lambda \in [0,1]$. The minimum value of the right-hand side of the last expression happens when $\lambda = \lambda( a^k )$, so:
\begin{equation*}
G_{a^k}^{z \ \prime}(0+) \geq - 2 \left[ 1 - \beta(a^k) \right] A(a^k) \left( 1 - \lambda( a^k ) \right) \left\| T(a^k) - a^k \right\| \left\| \widetilde{T}(a^k) - a^k \right\|. 
\end{equation*}
Using Remark \ref{remark: properties of Q} we conclude that:
\begin{equation*}
G_{a^k}^{z \ \prime}(0+) \geq - 2 \left[ 1 - \beta(a^k) \right] A(a^k) \left\| Q(a^k) - a^k \right\| \left\| \widetilde{T}(a^k) - a^k \right\|.
\end{equation*}
\qed
\end{proof}

Now we will prove an equivalence that characterizes the solution of (\ref{equation: constrained weber problem}) in terms of the iteration function $Q$. Moreover, if $x^*$ is a regular point that is not a vertex, then $x^*$ is a KKT point. 

From now on, let us consider that
\begin{equation}
\label{equation: characterization of omega}
\Omega = \left\{ y \in \mathbbm{R}^n : g(y) \leq 0, h(y) = 0 \right\},
\end{equation}
where $g : \mathbbm{R}^n \to \mathbbm{R}^s$ is a convex function and $h : \mathbbm{R}^n \to \mathbbm{R}^p$ is an affine function.

\begin{theorem} 
\label{theorem: equivalence minimum fixed point}
Let $\Omega$ be defined as in (\ref{equation: characterization of omega}) and $x \in \Omega$. Consider the following propositions:
\begin{enumerate}
\item[(a)] $x$ is a KKT point.
\item[(b)] $x$ is the minimizer of the problem (\ref{equation: constrained weber problem}).
\item[(c)] $Q(x) = x$.
\end{enumerate}

If $x \neq a^1, \ldots, a^m$, $g$ and $h$ are continuously differentiable, and $x$ is a regular point, then (a), (b) and (c) are equivalent.

If $x = a^k$ for some $k = 1, \ldots, m$, then (b) implies (c).
\end{theorem}

\begin{proof}
Let $x \neq a^1, \ldots, a^m$ be. Since $f$ is strictly convex and $\Omega$ is convex, the KKT optimality conditions are necessary and sufficient. Therefore, it holds that (a) is equivalent to (b).

Now we will prove that (b) implies (c). Let us suppose that $x$ is the minimizer of the problem (\ref{equation: constrained weber problem}). If $x$ were not a fixed point of the iteration function $Q$, we would have that $x \neq Q(x)$, which means that $f(Q(x)) < f(x)$ by Theorem \ref{theorem: descent property}. This contradicts the hypothesis.

To demonstrate that (c) implies (a), we will assume that $x$ is a fixed point of $Q$, that is, $x = Q(x)$. Since $Q(x) = P_{\Omega} \circ T(x)$, $x$ is the solution of:
\begin{equation*}
\begin{array}{rl}
\displaystyle \mathop{ \mathrm{argmin} }_z & \displaystyle F(z) = \frac{1}{2} \left\| z - T(x) \right\|^2 \\ \mathrm{subject \, to} & g(z) \leq 0, \\ & h(z) = 0.
\end{array}
\end{equation*}
Since $F$ and $g$ are convex, $h$ is affine, and $x$ is a regular point, the KKT optimality conditions hold at $x$. That is, there exist multipliers $\left\{ \mu_j \right\}_{j=1}^s$ and $\left\{ \lambda_j \right\}_{j=1}^p$ such that (see \cite{luenberger2008,quarteroni2006numerical}):
\begin{eqnarray*}
x - T(x) + \sum_{j=1}^s \mu_j \nabla g_j(x) + \sum_{j=1}^p \lambda_j \nabla h_j(x) & = & 0,
\\
\mu_j g_j( x ) & = & 0, \quad j = 1, \ldots, s,
\\
\mu_j & \geq & 0, \quad j = 1, \ldots, s,
\\
g(x) & \leq & 0,
\\
h(x) & = & 0.
\end{eqnarray*}
Multiplying these equations by $2 A(x)$, using equation (\ref{equation: nabla equal to minus widetilder}), Lemma \ref{lemma: widetilder equal to 2A(widetildet - x)} and Remark \ref{remark: properties of beta}, we obtain:
\begin{eqnarray*}
\nabla f(x) + \sum_{j=1}^s \left( 2 A(x) \mu_j \right) \nabla g_j(x) + \sum_{j=1}^p \left( 2 A(x) \lambda_j \right) \nabla h_j(x) & = & 0,
\\
\left( 2 A(x) \mu_j \right) g_j( x ) & = & 0, \quad j = 1, \ldots, s,
\\
\left( 2 A(x) \mu_j \right) & \geq & 0, \quad j = 1, \ldots, s,
\\
g(x) & \leq & 0,
\\
h(x) & = & 0.
\end{eqnarray*}
where $\left\{ 2 A(x) \mu_j \right\}_{j=1}^s$ and $\left\{ 2 A(x) \lambda_j \right\}_{j=1}^p$ are multipliers. Therefore, $x$ is a KKT point of the problem (\ref{equation: constrained weber problem}) (see \cite{luenberger2008,quarteroni2006numerical}).

Now, let us suppose that $x = a^k$ for some $k = 1, \ldots, m$. As before, if $x$ is a minimizer of the problem (\ref{equation: constrained weber problem}), then $Q(a^k) = a^k$, otherwise $f(Q(a^k)) < f(a^k)$, which would be a contradiction. \qed
\end{proof}

\begin{section}{Numerical experiments.} 
\label{section: numerical experiments}
The purpose of this section is to discuss the efficiency and robustness of the proposed algorithm versus a solver for nonlinear programming problems.

A prototype code of Algorithm \ref{algorithm: proposed algorithm} was programmed in MATLAB (version R2011a) and executed in a PC running Linux OS, Intel(R) Core(TM) i7 CPU Q720, 1.60GHz. 

We have considered a closed and convex set $\Omega \subset \mathbbm{R}^2$ defined by the set $\Omega = \left\{ y \in \mathbbm{R}^n : g(y) \leq 0 \right\}$, where $g$ is given by:
\begin{equation*}
g( x ) = \left[ \begin{array}{c} \displaystyle - 4 - \frac{ 1 }{ 8 } x + \frac{ 7 }{ 72 } x^2 + \frac{ 1 }{ 216 } x^2 ( x - 3 ) + y \\[3mm] \displaystyle \frac{ 4 }{ 5 } x + y - \frac{ 59 }{ 10 } \\[3mm] \displaystyle x - \frac{ 11 }{ 2 } \\[3mm] \displaystyle \frac{ 3 }{ 2 } x - y - \frac{ 35 }{ 4 } \\[3mm] \displaystyle x - y - \frac{ 13 }{ 2 } \\[3mm] \displaystyle - 4 + \frac{ 1 }{ 8 } ( x - 1 ) + \frac{ 1 }{ 16 } ( x - 1 )^2 + \frac{ 1 }{ 32 } ( x - 1 )^2 ( x - 3 ) - y \\[3mm] \displaystyle - \frac{ 1 }{ 3 } x - y - \frac{ 11 }{ 3 } \\[3mm] \displaystyle - \frac{ 2 }{ 3 } x - y - \frac{ 13 }{ 3 } \\[3mm] \displaystyle - 4 x + y - 19 \end{array} \right].
\end{equation*}
The feasible set is defined by linear and nonlinear constraints, as it can be seen in Figure \ref{figure: feasible set}.
\begin{center}
\begin{figure}
\includegraphics[width=115mm]{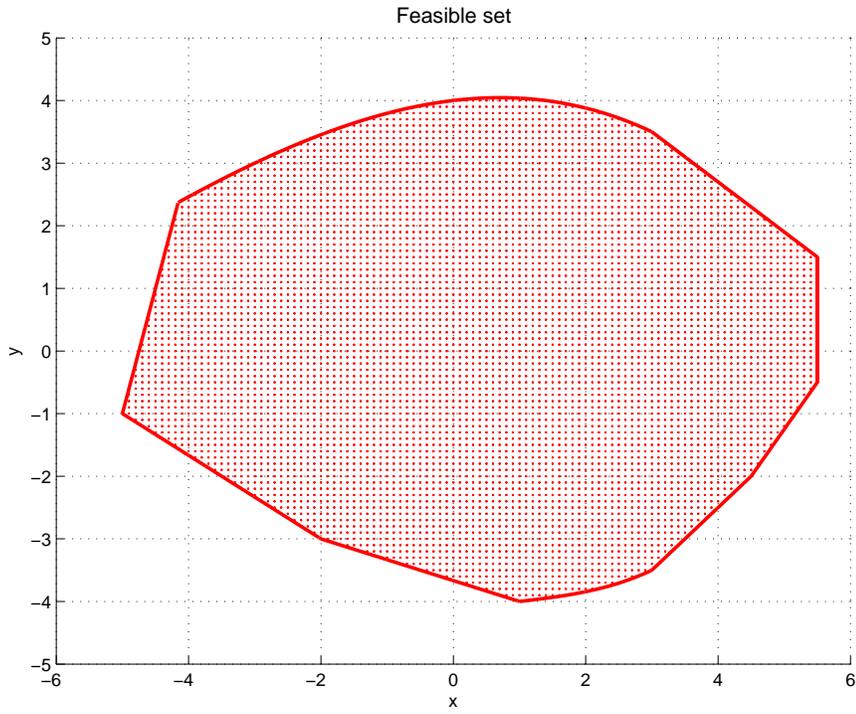}
\caption{Feasible set $\Omega$}
\label{figure: feasible set}
\end{figure}
\end{center}

We have built $1000$ different experiments where for each one:
\begin{itemize}
\item The number of vertices was $m = 50$.
\item The vertices were normally distributed random vectors, with mean equal to $0$ and standard deviation equal to $10$. 
\item The weights were uniformly distributed random positive numbers between $0$ and $10$.
\item Tolerance was set to $\varepsilon = 0.00001$.
\end{itemize}

On one hand, each experiment was solved using Algorithm \ref{algorithm: proposed algorithm} and, on the other hand, it was considered as a nonlinear programming problem and solved using function $fmincon$ (see \cite{fmincon} and references therein). Since the Weber function (\ref{equation: weber function}) is not differentiable at the vertices, nonlinear programming solvers may fail.

Let $x_m(i)$ be the solution of (\ref{equation: constrained weber problem}) obtained by $fmincon$ in experiment $i$, and $f_{m}(i) = f(x_m(i))$. Analogously, let $x_p(i)$ be the solution of (\ref{equation: constrained weber problem}) obtained by Algorithm \ref{algorithm: proposed algorithm} in experiment $i$, and $f_{p}(i) = f(x_p(i))$. Figure \ref{figure: functionalvalues} shows the difference between the arrays $f_{m}$ and $f_{p}$. Both methods finished succesfully in all cases, however, Algorithm \ref{algorithm: proposed algorithm} found equal or better results for all experiments. For example, the difference $f_{m} - f_{p}$ was greater than $0.01$ in $35$ experiments (the maximum difference ocurred in experiment $506$).
\begin{center}
\begin{figure}
\includegraphics[width=115mm]{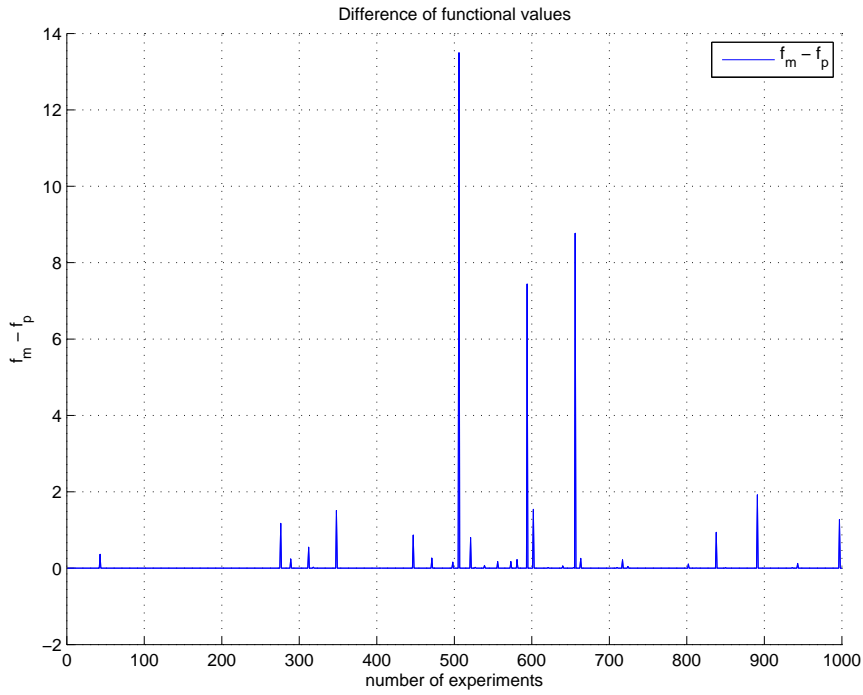}
\caption{Difference between minimum values found by Algorithm \ref{algorithm: proposed algorithm} and $fmincon$.}
\label{figure: functionalvalues}
\end{figure}
\end{center}

Feasibility of the solutions $x_p(i)$ can be checked computing $\max(g(x_m(i)))$. Results can be seen in Figure \ref{figure: feasibility}
\begin{center}
\begin{figure}
\includegraphics[width=115mm]{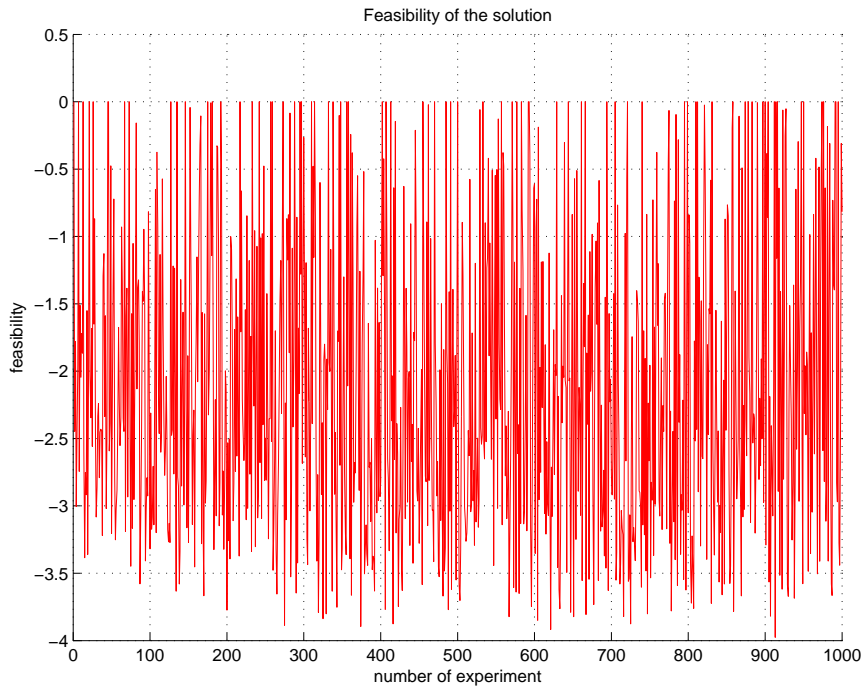}
\caption{Feasibility of the solution $x_p(i)$ obtained by Algorithm \ref{algorithm: proposed algorithm}.}
\label{figure: feasibility}
\end{figure}
\end{center}

\end{section}

\section{Conclusions} 
\label{section: conclusions}
This paper proposes a Weiszfeld-like algorithm for solving the Weber problem constrained to a closed and convex set, and it is well defined even when an iterate is a vertex. The algorithm consists of two stages: first, iterate using the fixed point modified Weiszfeld iteration (\ref{equation: iteration function of the modified algorithm with beta}), and second, either project onto the set $\Omega$ when the iterate is different from the vertices, or, if the iterate is a vertex $a^k$, take the point belonging to the line that joins $T(a^k)$ with $a^k$ as defined in (\ref{equation: iteration function of the proposed algorithm}). 

It is proved that the constrained problem (\ref{equation: constrained weber problem}) has a unique solution. Besides that, the definition of the iteration function $Q$ allows us to demonstrate that the proposed algorithm produces a sequence $\left\{ x^{(l)} \right\}$ of feasible iterates. Moreover, the sequence $\left\{ f \left( x^{(l)} \right) \right\}$ is not increasing, and when $x^{(l)} \neq Q\left( x^{(l)} \right)$, the sequence decreases at the next iterate. It can be seen that if a point $x^*$ is the solution of the problem (\ref{equation: constrained weber problem}) then $x^*$ is a fixed point of the iteration function $Q$. Even more, if $x^*$ is different from the vertices, the fact of being $x^*$ a fixed point of $Q$ is equivalent to the fact that $x^*$ satisfies the KKT optimality conditions, and equivalent to the fact that $x^*$ is the solution of the problem (\ref{equation: constrained weber problem}). These properties allows us to connect the proposed algorithm with the minimization problem. 

Numerical experiments showed that the proposed algorithm found equal or better solutions than a well-known standard solver, in a practical example with 1000 random choices of vertices and weights. That is due to the fact that the proposed algorithm does not use of the existence of derivatives at the vertices, because the Weber function is not differentiable at the vertices.


\end{document}